\documentclass[a4paper,12pt]{article}
\usepackage{times, url}
\usepackage{amsmath,amssymb,amsthm,amsfonts}

\newtheorem{theorem}{Theorem}[section]

\newtheorem{lemma}[theorem]{Lemma}
\newtheorem{con}{Conjecture}

\newtheorem{proposition}[theorem]{Proposition}

\begin{document}
\setcounter{page}{1}

\begin{center}
{\Large\bf A Study of @-numbers.}
\vspace{3mm}

{ \bf Abiodun E. Adeyemi}

Department of Mathematics, University of Ibadan, \\ 
Ibadan, Oyo state, Nigeria \\
e-mail: \url{elijahjje@yahoo.com}
\vspace{2mm}

\end{center}
\vspace{0mm}

\begin{abstract}
This paper deals more generally with @-numbers defined as follows: Call `\textit{alpha number}' of order $(\underline{\alpha},\bar{\alpha})\in\mathbb{H}^2$, (denote its family by @$_{(\underline{\alpha},\bar{\alpha})\in\mathbb{H}^2; \mathcal{A}\subset \mathbb{N}}$) any $n\in\mathcal{A}\subset \mathbb{N}$ satisfying $\sigma_{\underline{\alpha}}(n) = \alpha n^{\bar{\alpha}}$ where $\sigma_{\underline{\alpha}}(n)$ is sum of divisors function and $\alpha\in\mathbb{H}$, the set of \textit{quaternions}. Specifically, if integer $n$ is such that $\alpha=\alpha_1/\alpha_2,\ \alpha_1,\alpha_2\in\mathbb{Z}^+$ with $1\leq\max(\alpha_1, \alpha_2) \le \omega(n),$ $\ \le \tau (n), \ < n$ (where $\omega(n)$ is the number of distinct prime factors of $n$, $\tau (n)$ is the number of factors of $n$), then $n$ is respectively called strong, weak or very weak alpha number. We give some examples and conjecture that there is no odd strong alpha number of order $(1,1)$. The truthfulness of this assertion implies that there is no odd perfect and certain odd multi-perfect numbers. We give all the strong even alpha numbers of order $(1,1)$ below $10^5$ and then show that there is no odd strong alpha number of order $(1,1)$ below $10^5$, using some of our results motivated by some results of Ore and Garcia. With computer search this bound can easily be surpassed. In this paper, using Rossen, Schonfield and Sandor's inequalities, in addition to the aforementioned definition, we also bound the quotient $\alpha_1/\alpha_2 =\alpha$ of order $(1,1)$, though a very weak bound. Some areas for future research are also pointed out as recommendations.
\end{abstract} 
{\bf Keywords:} Perfect numbers, multi-perfect numbers, harmonic divisors number.\\
{\bf 2010 Mathematics Subject Classification:} 11N25, 11Y70.

\section{Introduction}
\qquad Throughout, let $\sigma_x (n)$, $\omega(n)$ and $\tau
(n)$ represent the sum of divisors function of $n$, the number of its distinct prime divisors, and the number of its distinct positive divisors, respectively. We define, for every positive integer $n$ and quaternion $x$, $$\sigma_x (n):=\sum_{d\in\mathbb{N}, d\mid n}d^x$$ and recall the definitions $\omega (n):=\sum_{ p \ prime, \\\ p\mid n}1 \ \ \ and \ \ \tau(n):=\sum_{d\in\mathbb{N}, d\mid n} 1.$
Note that traditionally, when $x=1$, we drop the subscript and simply write $\sigma (n)$ (called the sigma function) which now represents the sum of the factors of $n$ including $n$ itself. For example, the sum of positive divisors of $n=p^{\alpha}$ is $\sigma(p^\alpha)=1+p+...+p^\alpha$ while its $\omega(n)=1$, and $\tau(n)=\alpha +1$. Also, recall that any multiply perfect number $n$ satisfies $\sigma(n)=kn$ (where case $k=2$ specifically defines the perfect numbers such as 6 and 28, detailed in \cite{san1}, \cite{san2}) while any Ore's harmonic number \cite{{ore1}} satisfies $$\sigma(n)=(n\tau(n))/H(n),$$ where $H(n)$ is the harmonic mean of integer $n$. Now, we define the following numbers which extend both the multiply perfect numbers and Ore's harmonic numbers: Let $n\in\mathcal{A}\subset\mathbb{N}$ satisfy \begin{equation}\label{eqn1}
\sigma_{\underline{\alpha}}(n) = \frac{\alpha_1}{\alpha_2}n^{\bar{\alpha}}
\end{equation}
where $(\alpha_1 ,\alpha_2 )$ is a pair of arbitrary but co-prime integers. Then, we call $n$: 
\begin{itemize}
\item[(i)]{a strong alpha number} of order $(\underline{\alpha},\bar{\alpha})$” and denote its family by $@_{(\underline{\alpha},\bar{\alpha})\in\mathbb{H}^2;\mathcal{A}\subset\mathbb{N}}$, if $n$ further satisfies $2\leq\max ( \alpha_1 ,\alpha_2 )\leq\omega(n)$.
\item[(ii)]{a weak alpha number} of order $(\underline{\alpha},\bar{\alpha})$” and denote its family by $\bar{@}_{(\underline{\alpha},\bar{\alpha})\in\mathbb{H}^2;\mathcal{A}\subset\mathbb{N}}$, if $n$ further satisfies $2\leq\omega(n) < \max ( \alpha_1 ,\alpha_2 )\leq\tau(n)$.
\item[(iii)]{a very weak alpha number} of order $(\underline{\alpha},\bar{\alpha})$” and denote its family by $\bar{\bar{@}}_{(\underline{\alpha},\bar{\alpha})\in\mathbb{H}^2;\mathcal{A}\subset\mathbb{N}}$, if $n$ further satisfies $2\leq\tau(n) < \max (\alpha_1 ,\alpha_2) < n$. 
\end{itemize}
In continuation, suppose $n\in\mathcal{A}\subset\mathbb{N}$ satisfies, instead of (1),
\begin{equation}\label{eqn1}
\lfloor|\sigma_{\underline{\alpha}}(n)|\rfloor = \frac{\alpha_1}{\alpha_2}\lfloor |n^{\bar{\alpha}}|\rfloor
\end{equation}
where $\gcd(\alpha_1 ,\alpha_2)=1$, $\lfloor . \rfloor$ is the usual floor function and $| . |$ is the modulus function. Then, we call $n$:  
\begin{itemize}
\item[(i)]{a strongly floored alpha number} of order $(\underline{\alpha},\bar{\alpha})$” and denote its family by $@^{\lfloor \rfloor}_{(\underline{\alpha},\bar{\alpha})\in\mathbb{H}^2;\mathcal{A}\subset\mathbb{N}}$, if $n$ further satisfies $2\leq\max(\alpha_1, \alpha_2)\leq\omega(n)$.
\item[(ii)]{a weakly floored alpha number} of order $(\underline{\alpha},\bar{\alpha})$” and denote its family by $\bar{@}^{\lfloor \rfloor}_{(\underline{\alpha},\bar{\alpha})\in\mathbb{H}^2;\mathcal{A}\subset\mathbb{N}}$, if $n$ further satisfies $2\leq\omega(n) < \max(\alpha_1, \alpha_2)\leq\tau(n)$.
\item[(iii)]{a very weakly floored alpha number} of order $(\underline{\alpha},\bar{\alpha})$” and denote its family by $\bar{\bar{@}}^{\lfloor \rfloor}_{(\underline{\alpha},\bar{\alpha})\in\mathbb{H}^2;\mathcal{A}\subset\mathbb{N}}$, if $n$ further satisfies $2\leq\tau(n) < \max (\alpha_1 ,\alpha_2) < n$. 
\end{itemize}
And if $n\in\mathcal{A}\subset\mathbb{N}$ satisfies, instead of (1),
\begin{equation}\label{eqn1}
\lceil|\sigma_{\underline{\alpha}}(n)|\rceil = \frac{\alpha_1}{\alpha_2}\lceil |n^{\bar{\alpha}}|\rceil
\end{equation}
where $(\alpha_1 ,\alpha_2 )=1$, $\lceil . \rceil$ is the usual ceiling function and $| . |$ is the modulus function. Then, we call $n$: 
\begin{itemize}
\item[(i)]{a strongly ceiled alpha number} of order $(\underline{\alpha},\bar{\alpha})$” and denote its family by $@^{\lceil \rceil}_{(\underline{\alpha},\bar{\alpha})\in\mathbb{H}^2;\mathcal{A}\subset\mathbb{N}}$, if $n$ further satisfies $2\leq\max(\alpha_1, \alpha_2)\leq\omega(n)$.
\item[(ii)]{a weakly ceiled alpha number} of order $(\underline{\alpha},\bar{\alpha})$” and denote its family by $\bar{@}^{\lceil \rceil}_{(\underline{\alpha},\bar{\alpha})\in\mathbb{H}^2;\mathcal{A}\subset\mathbb{N}}$, if $n$ further  satisfies $2\leq\omega(n) < \max(\alpha_1, \alpha_2)\leq\tau(n)$.
\item[(iii)]{a very weakly ceiled alpha number} of order $(\underline{\alpha},\bar{\alpha})$” and denote its family by $\bar{\bar{@}}^{\lceil \rceil}_{(\underline{\alpha},\bar{\alpha})\in\mathbb{H}^2;\mathcal{A}\subset\mathbb{N}}$ if $n$ further satisfies $2\leq\tau(n) < \max(\alpha_1, \alpha_2) < n$. 
\end{itemize}
However, if $n\in\mathcal{A}\subset\mathbb{N}$ satisfies $\sigma_{\underline{\alpha}}(n) = \alpha n^{\bar{\alpha}}$ instead of (1), (2) and (3), we call $n$ \textit{a partial alpha number} of order $(\underline{\alpha},\bar{\alpha})$” and denote its family by $@^\star_{(\underline{\alpha},\bar{\alpha})\in\mathbb{H}^2;\mathcal{A}\subset\mathbb{N}}$.\\
Note that the major interest yet in number theory is the existence of even and odd of such numbers, as the application of special numbers is still a puzzle, though a golden challenge in Mathematics (see \cite{rao},\cite{san2}, \cite{tat}). Truly, even and odd alpha numbers exist, as we shall soon see various examples in this paper; but no odd strong alpha number, in particular, of order $(1,1)$ is known. So a main conjecture in this paper states that there is no such odd number. The truthfulness of this conjecture implies the odd perfect number conjecture and also implies that there is no certain odd multiply perfect numbers and Ore's harmonic number.
The following are formerly the conjectures which generalize the odd perfect number conjecture and other related conjectures concerning multiply perfect numbers: 
\begin{con}
$@_{(1,1); 2\mathbb{N}+1}=\emptyset.$
\end{con}
\begin{con}
Let $N_e(x,n)$, $\bar{N}_e(x,n)$ and $\bar{\bar{N}}_e(x,n)$ respectively count the number of strong, weak and very weak even alpha numbers $n$ that do not exceed real $x$. Then $N_e(x,n)\rightarrow\infty$, $\bar{N}_e(x,n)\rightarrow\infty$, and $\bar{\bar{N}}_e(x,n)\rightarrow\infty$ as $x\rightarrow\infty$.
\end{con}
\begin{con}
Let $\bar{N}_o(x,n)$ and $\bar{\bar{N}}_o(x,n)$ respectively count the number of weak and very weak odd alpha numbers $n$ that do not exceed real $x$. Then $\bar{N}_o(x,n)\rightarrow\infty$ and $\bar{\bar{N}}_o(x,n)\rightarrow\infty$ as $x\rightarrow\infty$.
\end{con}

 In what follows, we begin with some examples of @-numbers.
\section{ Examples of @-numbers}
Here, we extract some @-numbers from the table of the sum of positive divisors function given below:
\\
\begin{tabular}{|c|c|c|c|c|c|c|c|}
\hline 
$n$ & $\sigma_0(n)=\tau(n)$ & $\sigma_1(n)$ & $\sigma_2(n)$ & $\sigma_{0.5}(n)$ & $\sigma_{\sqrt{-1}}(n)$ & $\lfloor |\sigma_{0.5}(n) |\rfloor$ & $\lfloor |\sigma_{\sqrt{-1}}(n) |\rfloor$ \\ 
\hline 
1 & 1 & 1 & 1 & 1 & 1 & 1 & 1 \\ 
\hline 
2 & 2 & 3 & 5 & 2.4142 & 1.7692+ 0.6390i & 2 & 1 \\    
\hline 
3 & 2 & 4 & 10 & 2.7321 &  1.4548+ 0.8906i & 2 & 1 \\  
\hline 
4 & 3 & 7 & 21 & 4.4142 &  1.9527+ 1.6220i & 4 & 2 \\   
\hline 
5 & 2 & 6 & 26 &  3.2361 &  0.9614+ 0.9993i & 3 & 1 \\   
\hline 
6 & 4 & 12 & 50 & 6.5959 & 2.0049+ 2.5052i & 6 & 3 \\  
\hline 
7 & 2 & 8 & 50 & 3.6458 & 0.6336+ 0.9305i & 3 & 1 \\ 
\hline 
8 & 4 & 15 & 85 & 7.2426 & 1.466+ 2.4954i & 7 & 2 \\ 
\hline 
9 & 3 & 13 & 91 & 5.7321 & 0.8686+ 1.7007i & 5 & 1 \\ 
\hline 
10 & 4 & 18 & 130 & 7.8126 & 1.0624+ 2.3822i & 7 & 2 \\ 
\hline 
$\vdots$ & $\vdots$ & $\vdots$ & $\vdots$ & $\vdots$ & $\vdots$ & $\vdots$ & $\vdots$ \\ 
\hline
24 & 8& 60 & 850 & 19.787 & -0.0899+ 4.936i & 19 & 4 \\ 
\hline 
25 & 3 & 31 & 651 & 8.236  & -0.0356+ 0.922i & 8 & 0 \\ 
\hline 
26 & 4 & 42 & 850 & 11.118 & -0.0623+1.068i & 11 & 1 \\ 
\hline 
27 & 4 & 40 & 820 & 10.928 & -0.1200+ 1.547i & 10 & 1 \\ 
\hline 
28 & 6 & 56 & 1050 & 16.03 & -0.2719+ 2.845i & 16 & 2 \\ 
\hline 
29 & 2 & 30 & 842 & 6.385 & 0.0025-0.224i & 6 & 0 \\ 
\hline 
30 & 8 & 72 & 1300 & 21.344  & -0.5759+ 4.412i & 21 & 4 \\ 
\hline 
$\vdots$ & $\vdots$ & $\vdots$ & $\vdots$ & $\vdots$ & $\vdots$ & $\vdots$ & $\vdots$ \\ 
\hline
\end{tabular} 

\begin{center}
Table 1.
\end{center}

\textbf{Examples of strong alpha numbers of order }$(1,1)$:\\
\\
\begin{tabular}{|c|c|c|c|c|}
\hline 
$@_{(1,1)}$ & $\sigma_1(n)$ & $\alpha_1$ & $\alpha_2$ & $\omega(n)$ \\  
\hline 
$n = 6=2\cdot 3$& 12 & 2 & 1 & 2 \\ 
\hline 
$n = 28=2^2\cdot 7$& 56 & 2 & 1 & 2 \\ 
\hline 
$n = 523776 =2^9\cdot 3\cdot 11\cdot 31$ & 1571328 & 3 & 1 & 4 \\ 
\hline
$n = 707840=2^8\cdot 5\cdot 7\cdot 79$ & 1962240 & 3 & 1 & 4 \\ 
\hline
\end{tabular} 
\begin{flushleft}
Table 2.
\end{flushleft}
\textbf{Examples of weak alpha numbers of order }$(1,1)$:\\
\\
\begin{tabular}{|c|c|c|c|c|c|}
\hline 
$\bar{@}_{(1,1)}$ & $\sigma_1(n)$  & $\alpha_1$ & $\alpha_2$ & $\omega(n)$ & $\tau(n)$ \\ 
\hline 
$n = 24=2^3\cdot 3$ & 60  & 5 & 2 & 2 & 8 \\ 
\hline
$n = 11172 = 2^2\cdot 3\cdot 7^2\cdot 19$ & 31920  & 20 & 7 & 4 & 36 \\ 
\hline
$n = 544635^\ast = 3^2\cdot 5\cdot 7^2\cdot 13\cdot 19$ & 1244860  & 16 & 7 & 5 & 72 \\ 
\hline
$n = 931095^\ast = 3^4\cdot 5\cdot 11^2\cdot 19$ & 1931160  & 56 & 27 & 4 & 60 \\ 
\hline
$n = 6517665 = 3^4\cdot 5\cdot 7\cdot 11^2\cdot 19$ & 15449280  & 64 & 27 & 5 & 120 \\ 
\hline
\end{tabular} 
\begin{flushleft}
Table 3.
\end{flushleft}
\begin{tiny}
Thanks to Professor J. Shallit for pointing our attention (in a private communication) to the examples of odd alpha numbers in asterisk on Table 3 which enabled us to refine the conjecture of this paper concerning the existence of odd alpha numbers.
\end{tiny}

\section{Results on @-numbers:}
Then, we proceed to offer the following results which were inspired by the work of  Garcia \cite{Ga} and Ore \cite{ore2}, and which shall find application later in searching for alpha numbers.
\begin{theorem}\label{t1} All perfect, multiply perfect and Ore's harmonic number are alpha numbers of order $(1,1)$.
 \end{theorem}
  \begin{proof} Let a positive integer $n$ be $k$-fold multi perfect if $\sigma(n)=kn$. Also observe that $\sigma(n)< nn$ for every integer $n>1$, since $\tau(n)<n$ and each $d_i\mid n$ is such that $1\le d_i \le n$ and $n>\min (d_1, ... d_z)$. Thus there is no $n$-fold multi perfect number. So, 
 Theorem 3.1 is an implication of the definitions of multiply perfect numbers, Ore's harmonic numbers and alpha numbers since every multiply perfect number and Ore's harmonic number satisfy (1).
 \end{proof}
 \begin{theorem}\label{t2} A power of prime can neither be strong, weak, nor very weak alpha number of order $(1,1)$.
 \end{theorem}
 \begin{proof}It suffices to show that there is no $n=p^\alpha$ satisfying $\dfrac{\sigma(n)}{n}=\dfrac{\alpha_1}{\alpha_2}$ where arbitrary integers $\alpha_1$ and $\alpha_2$ are co-prime, and $\alpha_1< n$. Note that, $\sigma(p^\alpha)=1+p+...+p^\alpha$ is co-prime to $p$, so to $p^\alpha$, therefore the fraction $\sigma(n)/n$ is reduced. Should this equal $\alpha_1/\alpha_2$, then $\sigma(n)$ divides $\alpha_1$. In turn, this leads to $\sigma(p^\alpha)=1+p+...+p^\alpha\le p^\alpha$, a contradiction.
\end{proof}
 \begin{theorem}\label{t3} No square-free odd integer is a strong or weak alpha number of order $(1,1)$.
 \end{theorem}
   \begin{proof} We show that $\max(\alpha_1, \alpha_2) \not<\tau(n)$ for any square-free odd integer $n$ satisfying (1).
On the contrary, suppose a positive odd integer $n$ of the canonical prime factorisation form $n=\prod_{i=1}^{\omega(n)}  p_i$ satisfies (1), then  $$\alpha_2=\frac{\alpha_1\prod_{i=1}^{\omega(n)}  p_i}{\sigma(\prod_{i=1}^{\omega(n)}  p_i)}=\frac{\alpha_1\prod_{i=1}^{\omega(n)}  p_i}{\prod_{i=1}^{\omega(n)}  \sigma(p_i)}=\frac{\alpha_1\prod_{i=1}^{\omega(n)}  p_i}{2^{\omega(n)}\prod_{i=1}^{\omega(n)} q_i}$$
where each $q_i$ is an integer greater than $1$, since $\sigma$ is multiplicative. Now that $\alpha_1$ and $\alpha_2$ are integers and each $p_i$ is an odd prime, it is necessary that $2^{\omega(n)} \mid \alpha_1$. Moreover, the $\min\lbrace q_1, q_2,...q_{\omega(n)}\rbrace$ must divide $\alpha_1$ since $1<\min\lbrace q_1, q_2,...q_{\omega(n)}\rbrace<\min\lbrace p_1, p_2,...p_{\omega(n)}\rbrace$. Thus, $\alpha_1> \tau(n)=2^{\omega(n)}$, a contradiction. Hence, Theorem 3.3.
\\
 \end{proof}
 An alternative proof of Theorem 3.3 can again proceed using the multiplicative property of sigma-function from $$\alpha_1=\frac{\alpha_2\sigma(\prod_{i=1}^{\omega(n)}  p_i) }{\prod_{i=1}^{\omega(n)}  p_i}=\alpha_2\cdot 2^{\omega(n)}\frac{\prod_{i=1}^{\omega(n)}\frac{\sigma(p_i)}{2}}{\prod_{i=1}^{\omega(n)}  p_i}$$ where one could infer that each $p_i\mid \alpha_2$ or $p_i\mid {\prod_{i=1}^{\omega(n)}{\sigma(p_i)}/{2}}\in\mathbb{N}$, by Euclid's lemma, and specifically $$\max \lbrace p_1, p_2,...p_{\omega(n)}\rbrace\nmid \prod_{i=1}^{\omega(n)}{\sigma(p_i)}/{2}\in\mathbb{N}.$$ Hence $\max \lbrace p_1, p_2,...p_{\omega(n)}\rbrace$ must divide $\alpha_2$ since $\alpha_1$ is an integer. Thus implies that $\alpha_1 > 2^{\omega(n)}=\tau(n)$, also a contradiction.
 \begin{theorem}\label{t4} There is no strong odd alpha number $n$ of order $(1,1)$ in the canonical form $\prod_{i=1}^{\omega(n)}p_i^{n_i}$, with each $n_i\ge 2$ and each $\sigma(p_i^{n_i})$ prime.
 \end{theorem}
 \begin{proof}
 It is sufficient to show that $\alpha_2 >\omega(n)$ for such $n$. Now, suppose contrary that such odd $n$ exists with its attached condition $n_i\ge 2$ and $\sigma(p_i^{n_i})$ is prime for each $i$. Thus
 \begin{equation}
 |\lbrace p_1, p_2, ..., p_{\omega(n)} \rbrace|= | \lbrace \sigma(p_1^{n_1}), \sigma(p_2^{n_2}), ..., \sigma(p_{\omega(n)}^{n_{\omega(n)}}) \rbrace|,
 \end{equation} where $|X|$ represent the cardinality of set $X$. Now by condition (4) and the fact that  $\omega(\sigma(p_i^{n_i}))=1$ implies $\omega(n)= \omega(\sigma(n))$, $n$ through (1), satisfies
 \begin{equation}
\alpha_2=\alpha_1\frac{\prod_{i=1}^{\omega(n)}p_i^{n_i}}{\prod_{i=1}^{\omega(n)}\sigma(p_i^{n_i})}\ge \min \lbrace p_1, ..., p_{\omega(n)}\rbrace^{\Omega(n) -\omega(\sigma(n))} \ge 3^{\Omega(n) -\omega(\sigma(n))}, 
\end{equation}
where $\Omega(n):=\sum_{ p \ prime, \\\ p^v\mid n}1$.\\
Thus (5) implies $$\alpha_2 \ge 3^{\Omega(n) -\omega(n)} \ge 3^{\omega(n)} >\omega(n),$$
since $n_i\ge 2 \ \forall \ i$. Thus yield the required contradiction.\\ 
 \end{proof}
 An alternative approach to establishing Theorem 3.4 is to consider $n_i\ge 2 \ \forall \ i$ which implies $\Omega(n)-\omega(n)\ge \frac{\Omega(n)}{2}$. This further implies $$\log_3\prod_{i=1}^{\omega(n)}3^{n_i -1}\ge \log_3\prod_{i=1}^{\omega(n)}3^{\frac{n_i}{2}}\ge \log_3\prod_{i=1}^{\omega(n)}n_i +1.$$
One can then conclude through 
$$\alpha_2=\alpha_1\frac{\prod_{i=1}^{\omega(n)}p_i^{n_i}}{\prod_{i=1}^{\omega(n)}\sigma(p_i^{n_i})}\ge 3^{\Omega(n) -\omega(n)} \ge \prod_{i=1}^{\omega(n)}3^{n_i -1}\ge\prod_{i=1}^{\omega(n)}3^{\frac{n_i}{2}}\ge \prod_{i=1}^{\omega(n)}n_i +1 =\tau(n)$$ which yields a contradiction, since necessarily $\max(\alpha_1, \alpha_2) < \omega(n)$.\\
\textbf{Remark}: An improvement of Theorem 3.4 is to exclude the condition $\sigma(p_i^{n_i})$ prime for every $i$ in the the statement of the theorem.
  \begin{theorem}\label{t5} Except for perfect numbers, no integer $n$ with prime factor decomposition $p_1^\alpha p_2$ is a strong alpha number of order $(1,1)$.
 \end{theorem}
 \begin{proof} Let $n=p_1^\alpha p_2$ with distinct primes $p_1$ and $p_2$ be a strong alpha number of order $(1,1)$. Then there exist $\alpha_1, \alpha_2\in\mathbb{Z}^+$ with  $(\alpha_1, \alpha_2)=1$ and $2\le \max(\alpha_1, \alpha_2)\le \omega(n)=2$ such that
 \begin{equation}
 \alpha_2=\frac{\alpha_1 \cdot p_1^\alpha\cdot p_2}{(p_1^\alpha +...+p_1 +1)(p_2+1)}. 
\end{equation}
Then we consider the possibilities $p_2=2$ and $p_2\not=2$:\\
For case $p_2=2$, it is necessary that $$\alpha_2=\frac{2\alpha_1\cdot p_1^\alpha}{3(p_1^\alpha +...+p_1 +1)}$$ and $p_1^\alpha +...+p_1 +1\mid 2\alpha_1$ since $(p_1^\alpha +...+p_1 +1, p_1^\alpha)=1.$ Thus follows that $p_1^\alpha +...+p_1 +1\le 2\alpha_1\le 4$ since $\alpha_1$ is $2$ at maximum. But $p_1^\alpha +...+p_1 +1\geq 3^\alpha +...+3 +1\not<4$. Thus necessarily $\alpha=1$, $p_1=3$, and consequently $n=6$, a perfect number.\\ 
For case $p_2\not=2$, from (6), we write  
\begin{equation}
 \alpha_2=\frac{p_1^\alpha\cdot p_2}{(p_1^\alpha +...+p_1 +1)((p_2+1)/\alpha_1)} 
\end{equation} and notice that the fact that $\sigma(p_1^\alpha p_2) > p_1^\alpha p_2$ implies, through (6), that $1\le \alpha_2< \alpha_1\le 2$, hence $\alpha_1=2$ and $\alpha_2=1$.\\
Equation (7) thus implies   
\begin{equation}
 p_2= p_1^\alpha +...+p_1 +1,
\end{equation}
since $p_2$ should necessarily factor $p_1^\alpha +...+p_1 +1$ while $p_1^\alpha +...+p_1 +1$ also divides $p_2$, a consequence of the fact that $(p_1^\alpha +...+p_1 +1, p_1^\alpha)=1$ and $(({p_2+1})/{2}, p_2)=1$.
This further implies 
$$ p_1^\alpha  =\frac{p_2+1}{2} =\frac{p_1^\alpha +...+p_1 +2}{2}\ge \frac{3^\alpha +...+ 3 +2}{2},$$ since $\alpha_2=1$ for (7). Thus $p_1\mid 2$, since necessarily $p_1\mid ({p_1^\alpha +...+p_1 +2})/{2}$. Thus implies $p_1=2$. Recalling Euclid's theorem which states that every number of the form $2^k\sigma(2^k)$ is perfect if $\sigma(2^k)$ is prime, it therefore follows that $n=p_1^\alpha p_2$ is a perfect number. Therefore, Theorem 3.4 holds.
 \\
 \end{proof}
  The next result under this section, although only provides a weak bound on the quotient $\alpha_1/\alpha_2$ of any alpha number $n$, is a direct consequence of the following important Lemmas connecting arithmetic functions such $\phi(n)$ and $\sigma(n)$:
\begin{lemma} \emph{(J.B. Rosser and L. Schenfield \cite{ros})}\\
If $n\geq 3$, then $\frac{n}{\phi(n)}<e^\gamma\log\log n +\frac{0.6483}{\log\log n}$ where $\gamma$ is the \emph{Euler constant}.
\end{lemma}
\begin{lemma} \emph{(J. s\'{a}ndor \cite{san})}\\
There is a constant $C>0$ such that $\frac{n}{\phi(n)}<C\cdot\log\log\phi(n) \ \forall \ n>3.$
\end{lemma} 
\begin{theorem}\label{t6} For any alpha number $n\ge 3$ of order $(1,1)$, $${\alpha_1}/{\alpha_2} < e^\gamma\log\log n +\frac{0.6483}{\log\log n},$$ 
where $\gamma$ is the \emph{Euler constant}, and moreover, there exists a constant $C>0$ such that $${\alpha_1}/{\alpha_2}< C\log\log \phi(n)$$.
 \end{theorem}
 \begin{proof}
It is sufficient to show that $\frac{\sigma(n)}{n}<\frac{n}{\phi(n)}$ using Euler product formula $\phi(n)=n\prod_{p\mid n}(1-\frac{1}{p})$ and $\sigma(n=\prod_{i=1}^{\omega(n)}p_i^{n_i})=\prod_{i=1}^{\omega(n)}\frac{p^{n_i+1}-1}{p-1}$, and then applying Lemma 3.6 and Lemma 3.7 into the definition of alpha numbers.
\\
 \end{proof}
 \par
\noindent
 \textbf{Remark :} We remark that the above bound is very weak for strong alpha number, since quotient ${\alpha_1}/{\alpha_2}\le \omega(n)$ necessarily, so it is presented for further study.
 \begin{theorem}\label{t7} Any strong alpha number $n$ of order $(\underline{\alpha}\in\mathbb{N},\bar{\alpha}\in\mathbb{N})$ with $\omega(n)\le 2$ cannot be a perfect square.
 \end{theorem}
 \begin{proof}
Case $\omega(n)=1$ is a direct implication of Theorem 3.2, so we only need to establish that for $\omega(n)=2$, $n=p_1^{2x}p_2^{2y}$ with distinct primes $p_1$ and $p_2$ does not satisfy
\begin{equation}
\alpha_2=\frac{\alpha_1 \cdot p_1^{2\bar{\alpha}x}\cdot p_2^{2\bar{\alpha}y}}{\sigma_{\underline{\alpha}}(p_1^{2x})\sigma_{\underline{\alpha}}(p_2^{2y})}, 
\end{equation}
where $\alpha_1=2$ and $\alpha_2=1$ (that is, $n$ contradicts the definition of strong alpha numbers). If on the contrary, equation (9) holds, then $n$ must satisfy $2p_1^{2\bar{\alpha}x}=\sigma{\underline{\alpha}}(p_2^{2y})$ and $p_2^{2\bar{\alpha}y}=\sigma{\underline{\alpha}}(p_1^{2x})$, or $2p_2^{2\bar{\alpha}y}=\sigma{\underline{\alpha}}(p_1^{2 x})$ and $p_1^{2\bar{\alpha}x}=\sigma{\underline{\alpha}}(p_2^{2y})$, since $(p_1^{2\bar{\alpha}x}, \sigma{\underline{\alpha}}(p_1^{2x}))=1$ and $(p_2^{2\bar{\alpha}y}, \sigma{\underline{\alpha}}(p_2^{2y}))=1$. But this contradicts the fact that both $\sigma{\underline{\alpha}}(p_1^{2x})$ and $\sigma{\underline{\alpha}}(p_2^{2y})$ are odd. Hence Theorem 3.9.
\\
 \end{proof}
 \vspace{0.5cm}
   \begin{theorem}\label{t8} $@_{(\underline{\alpha}\in\mathbb{R}_{>0},\bar{\alpha}\in\mathbb{R}_{>\underline{\alpha}+2});\mathbb{N}}=\emptyset$.
 \end{theorem}
\begin{proof} On the contrary, let $n\in@_{(\underline{\alpha}\in\mathbb{R}_{>0},\bar{\alpha}\in\mathbb{R}_{>2 +\underline{\alpha}});\mathbb{N}}\not=\emptyset$ such that $n$ with unique prime decomposition $\prod_{i=1}^{\omega(n)}p_i^{n_i}$ satisfies (1) with $2\le\max (\alpha_1,\alpha_2)\le \omega(n)$.  Thus, in order to establish Theorem 3.10, it is sufficient to show that
\begin{equation}
\prod_{i=1}^{\omega(n)}\sigma_{\underline{\alpha}}(p_i^{n_i})\le\prod_{i=1}^{\omega(n)}p_i^{(n_i\bar{\alpha}-1)} \ \forall \ \ i,
\end{equation} 
if $\underline{\alpha}\in\mathbb{R}_{>0}$ and $\bar{\alpha}\in\mathbb{R}_{>\underline{\alpha}+2}$, since $\sigma_x(n)$ is a multiplicative function and $n$ necessarily satisfies $\prod_{i=1}^{\omega(n)}p_i >\omega(n)$ and 
\begin{equation}
 \alpha_2=\alpha_1\frac{\prod_{i=1}^{\omega(n)}p_i^{(n_i\bar{\alpha}-1)} \prod_{i=1}^{\omega(n)}p_i}{\sigma_{\underline{\alpha}}(\prod_{i=1}^{\omega(n)}p_i^{n_i})}.  
\end{equation}
Then, we have to verify the following assertion which implies (10):
\begin{equation}
\sigma_{\underline{\alpha}}(p_i^{n_i})\le p_i^{(n_i\bar{\alpha}-1)} \ \forall \ i 
\end{equation}
where $\underline{\alpha}\in\mathbb{R}_{>0}$ and $\bar{\alpha}\in\mathbb{R}_{>\underline{\alpha}+2}$.
Clearly, by the definition of $\tau(.)$ and $\sigma_x(.)$, $$\sigma_{\underline{\alpha}}(p_i^{n_i})<p_i^{n_i\underline{\alpha}}\cdot \tau(p_i^{n_i}),$$ and obviously, $$p_i^{n_i\underline{\alpha}}\cdot \tau(p_i^{n_i})<p_i^{n_i\underline{\alpha}}\cdot 3^{n_i}\le p_i^{n_i\underline{\alpha}}\cdot p_i^{2n_i-1}.$$ Thus, setting $\bar{\alpha}=\underline{\alpha}+ 2$ in (12) implies (10). Consequently $\alpha_2 > \omega(n)$, a contradiction. Therefore Theorem 3.10 follows.\\  
\end{proof}
In what follows, we formally include the following propositions without proof:\\
\begin{proposition}
Let $N(x)$ and $N^\star(x)$ respectively count the number of strong and partial alpha numbers that do not exceed real $x$. Then $N^\star(x) \gg N(x)$ for a sufficiently large $x$.

\end{proposition}

\begin{proposition}
Let $N^{\lfloor  \rfloor}(x)$, $N^{\lceil  \rceil}(x)$ and $N^\star(x)$ respectively count the number of strongly floored, strongly ceiled and partial alpha numbers that do not exceed real $x$. Then $N^\star(x) \gg N^{\lfloor  \rfloor}(x)$ and $N^\star(x)\gg N^{\lceil  \rceil}(x)$ for a sufficiently large $x$.
\end{proposition}
\begin{proposition}
Let F-alpha number, be any alpha number (strong, weak or very weak) such that the ordered pair $(\alpha_1, \alpha_2)$ in the frame $\alpha_2[\sigma_{\underline{\alpha}}(n)] = {\alpha_1}[n^{\bar{\alpha}}]$ is a pair of invertible functions $(\mathfrak{f}_{\underline{\alpha}}(n), (\mathfrak{f}_{\bar{\alpha}}(n))$, then with $\mathfrak{f}_{\underline{\alpha}}(n)$ and $\mathfrak{f}_{\bar{\alpha}}(n)$ as secret keys and $n$ itself a product of public primes, a more secured encryption and decryption can be made feasible.
\end{proposition}
 \section{  \textbf{@- numbers below $10^5$}}
Here, we first obtain the even strong alpha numbers of order $(1,1)$ below $10^5$. In order to achieve this goal, we shall make use of an approach similar in nature to that of Ore \cite{ore2} and Garcia \cite{Ga}. By the fundamental theorem of arithmetic and definition of alpha numbers, it is necessary to consider $$n=2^\lambda x <10^5 : x = p_2^{n_2}p_3^{n_3}\cdots p_{\omega(n)}^{n_{\omega(n)}}$$  in the frame \begin{equation}
\alpha_2=\alpha_1\cdot \frac{n}{\sigma(n)}, \ \ 2 \le \max (\alpha_2, \alpha_1 )\le\omega(n),
\end{equation} where $p_2, p_3, ..., p_{\omega(n)}$ are distinct odd primes. By virtue of Lemmas 2, 3, 4, and 5, $\lambda \leq 13$ with $\omega(2^{13}x <10^5)< 4$. So we consider the possibility $2^{13}\mid n$ in (2): $$\alpha_2=\alpha_1\cdot \frac{2^{13}x}{3\cdot 43\cdot 127\sigma(x)}.$$ 
But again, the fact that $\alpha_1\leq \omega(2^{13}x <10^5)<127$ strictly implies that $127$ must divide $n$, and further implies that $n\geq 2^{13}\cdot 127>100000$, a contradiction. Thus, there is no even strong alpha number of order $1$ divisible by $2^{13}$. Continuing this way, the first $\lambda$ such that $2^\lambda\mid n < 100000$ is $6$, this gives us $$\alpha_2 =\alpha_1\cdot \frac{2^{6}x}{127\sigma(x)}.$$ Then, since $\alpha_1\leq \omega(2^{6}x<10^5)<127$ and $2^6\cdot 127^2 > 100000$, it is only possible that $127\mid n$ such that $$\alpha_2 =\alpha_1\cdot \frac{2^{6}\cdot 127y}{127\cdot 2^7\sigma(y)}=\alpha_1\cdot\frac{y}{2\sigma(y)}.$$
Now we are left with two possibilities: either integer $y=1$ such that strictly $n=2^{6}\cdot 127,$ $\alpha_1=2$ and $\alpha_2=1$, or odd $y\ge 3$, $n>2^{6}\cdot 127,$ and $2\mid \alpha_1$. The later case implies that $\alpha_1$ satisfies the following: 
\begin{align*}\label{E:1}
\alpha_1
&=2\alpha_2\cdot \frac{\sigma(y)}{y} : \alpha_2\in\mathbb{N},\\
\alpha_1
& =2t , \ t=1, 2, 3, ...,\\
\\
\alpha_1
& <\omega(n=2^6 \cdot 127y<10^5)< 4, \ \textit{since} \ 2^6\cdot 127\cdot 3\cdot 5 >10^5.\\
\end{align*}
Thus implies that $\alpha_1\not>2$. Hence $y=1$, and so $n=2^{6}\cdot 127$ is an alpha number with $\alpha_1=2$ and $\alpha_2=1$. Repeating the above procedure for $\lambda = 5, 4, 3, 2$ and $\lambda=1$ yields the following even strong alpha numbers of order $(1,1)$.\\
\\
\begin{tabular}{|c|c|c|c|c|}
\hline 
$@_{(1,1)}$ & $\sigma(n)$ & $\alpha_1$ & $\alpha_2$ & $\omega(n)$ \\  
\hline 
$n = 6=2\cdot 3$& 12 & 2 & 1 & 2 \\ 
\hline 
$n = 28=2^2\cdot 7$& 56 & 2 & 1 & 2 \\ 
\hline 
$n = 120 =2^3\cdot 3\cdot 5$ & 360 & 3 & 1 & 3 \\ 
\hline
$n = 496=2^4\cdot 31$ & 992 & 2 & 1 & 2 \\ 
\hline
$n = 672=2^5\cdot 3\cdot 7$ & 2016 & 3 & 1 & 3 \\
\hline
$n = 1090=2^3\cdot 3\cdot 5\cdot 7\cdot 13$ & 40320 & 4 & 1 & 5 \\  
\hline
$n = 8128=2^6\cdot 127$ & 16256 & 2 & 1 & 2 \\
\hline
$n = 30240=2^5\cdot 3^3\cdot 5\cdot 7$ & 120960 & 4 & 1 & 4 \\  
\hline
$n = 32760 =2^3\cdot 3^2\cdot 13\cdot 7\cdot 5$ & 131040 & 4 & 1 & 5 \\ 
\hline
\end{tabular} 
\begin{flushleft}
Table 3.
\end{flushleft}
\textbf{Remark}: We observed that all the strong even alpha numbers between $1$ and $100000$ as recorded in the above table are purely multiply perfect numbers, so the first non-multiply perfect strong even alpha number will be $\ge 10^5$, and will be an interesting example of alpha numbers.
\\ 
\textbf{Odd alpha numbers of order $(1,1)$ below $10^5$.}
\begin{theorem}
There is no strong odd alpha number of order $(1,1)$ below $10^5$.
\end{theorem}
Our strategy here is to make the above method for obtaining even strong alpha numbers rigorous, and to achieve this goal, we first define $\omega_B(x):= \omega(x <B)$, and then continue by introducing the following function to conveniently extract the prime-power divisors of any positive integer, say $y$, bounded by $B>0$:\\
\begin{align*}
{\mathit{f}}_i(y)|_B
&:=\left\{
	\begin{array}{ll}
		y, \ \emph{ if } \ y \ \emph{is} \ \emph{prime} \ \emph{ and } \ y > {\omega_B(yx)}.\\
		\\
p_j^{n_j}: j= (i-1) + k, \ p_k = \min (p_1,p_2,...,p_{\omega(y)}) > {\omega_B(p_i^{n_i}x)} \ \emph{where} \ y=\prod_{i=1}^{\omega(x)} p_i^{n_i} \\ \emph{with} \ p_i < p_{i+1} \ \forall \ i<\omega(x).\\
		\\
		1, otherwise.\\
	\end{array}
\right.\\
\end{align*}

Then, we define the special integers that are potential alpha numbers not greater than bound $B>0$, in that they satisfy certain conditions of alpha numbers and can further be tested. These set of integers shall henceforth be regarded as `\textit{virtual alpha number}'.\\
\\
\textbf{Definition.} Let `$\mid\mid $' be such that $p^y\mid\mid x\Rightarrow p^y\mid x$ but $p^{y+1}\nmid x$. Further let $p_{\star}$ be prime such that $p_{\star}^{\lambda_\star}\mid \mid u\in\mathbb{N}\Rightarrow u\mid \bar{n}\in\mathbb{N}$, and if another prime $p_i\mid \bar{n}$, prime $p_\star<p_i$, then, we say that $$u=p_{\star}^{\lambda_\star}\cdot \prod_ {i\in I=\lbrace 1, ... , k\rbrace}({\mathit{f}}_i(\sigma(p_{\star}^{\lambda_\star})))^{\alpha_i}\le B$$ 
is the alpha seed of length $k$ of \textit{virtual alpha number} $\bar{n}= u\prod_{({f}_j(\sigma(u)),u)=1} {\mathit{f}}_j(\sigma(u))$ bounded by $B$ where $p_{\star}^{\lambda_\star}$ is the \textit{generator} of $u$.\\
\\
Thus alpha number $n$ relates directly with virtual alpha number $\bar{n}$, in one and only one way, as follows: $n\ge\bar{n}$. Hence, every $n<\bar{n}$ is clearly not an alpha number.
Thus (regarding every virtual alpha number $\bar{n}$ that make alpha number $n$ $:n=\bar{n}\prod p$ as being transformed), the following characteristic function is necessary and sufficient to determine the virtual alpha numbers $\bar{n}$ bounded by real $B>0$ that transformed to alpha number $n$. 
 \begin{align*}
\chi_\alpha(\bar{n})
&:=\left\{
	\begin{array}{ll}
		0; \ \emph{if it is necessary that} \  p_i^k\mid n \ \emph{where} \ p_i<p_\star \ \emph{and} \  k\ge 1, \ \emph{or if} \ \bar{n}>B.\\
		\\
		1; \ \emph{if} \ \emph{otherwise}.\\
	\end{array}
\right.
\end{align*} 
\\
\textbf{Proof of Theorem 4.1}
\begin{proof} On the contrary, suppose such odd alpha number $n<10^5$ exists. Then, by virtue of Theorem 3.2 to Lemma 3.5, there exists alpha seeds of generator $p_\star^{\lambda_\star}$ with $p_\star\in\lbrace 3, 5, 7, 11, 13\rbrace$, since $10^5<17^2\cdot 19\cdot 23< 17^2\cdot 19^3 <17\cdot 19^2\cdot 23 < ...$ Thus, we have the following table:
\\
\begin{center}
\begin{tabular}{|c|c|c|c|c|c|}
\hline 
$p_\star$ & $3$ & $5$&$7$ &$11$ &$13$ \\ 
\hline 
$\lambda_\star$ & $7$  & $4$ & $3$ &$2$ &$2$\\ 
\hline 
$\omega_{10^5}(p_\star^{\lambda_\star}x)$ & $3$ & $3$ & $3$ &$3$ &$3$ \\ 
\hline 

\end{tabular}
\\ 
\begin{flushleft}
Table 4.
\end{flushleft}
\end{center} 
To make our work easier, we first obtain the sum of divisors of some higher powers of primes that are involved in our investigation: 
\begin{center}
\begin{tabular}{|c|c|}
\hline 
$x$ & $\sigma(x)$ \\
\hline 
$3^7$ & $2^4\cdot 541$ \\  
\hline 
$3^6$ & $1093$ \\ 
\hline 
$3^5$ & $2^2\cdot 7\cdot 13$ \\ 
\hline 
$3^4$ & $11^2$ \\ 
\hline 
$3^3$ & $2^3\cdot 5$ \\ 
\hline 
$3^2$ & $13$ \\ 
\hline 
$5^5$ & $3^3\cdot 7\cdot 31$ \\ 
\hline 
$5^4$ & $11\cdot 71$ \\ 
\hline 
$5^3$ & $2^2\cdot 3\cdot 13$ \\ 
\hline 
$5^2$ & $ 31$ \\ 
\hline 
$7^4$ & $2801$ \\ 
\hline 
$7^3$ & $2^4\cdot 5^2$ \\ 
\hline 
$7^2$ & $3\cdot 19$ \\ 
\hline 
$11^2$ & $7\cdot 19$ \\ 
\hline 
$13^3$ & $2\cdot 5\cdot 7\cdot 17$ \\ 
\hline 
$13^2$ & $3\cdot 61$ \\ 
\hline 
$19^2$ & $3\cdot 127$\\ 
\hline 
$31^2$ & $3\cdot 331$\\ 
\hline 
\end{tabular} 
\begin{flushleft}
Table 5.
\end{flushleft}
\end{center}
Then we construct, using Table 5 and equation (3), a consequence of (1),the strong alpha seeds of every odd integer $<10^5$.\\ In the table, we denote the set of alpha seeds of length $2$ and bound $B>0$, by \textit{$\mathit{U}(p_{\star}^{\lambda_\star}, B, 2):=\lbrace u: u<B \ \emph{is an alpha seed}\rbrace.$ Note also that the letters $n$ and $nn$ within the table mean `none' and `not necessary' respectively}\\
\begin{center}
\begin{tabular}{|c|c|c|c|c|c|c|}
\hline 
 $p_{\star}^{\lambda_\star}$ & ${f}_1(\sigma(p_{\star}^{\lambda_\star}))=s_1$& ${f}_2(\sigma(p_{\star}^{\lambda_\star}))=s_2$ &${f}_1(\sigma(s_1))$ &${f}_2(\sigma(s_1))$& ${f}_1(\sigma(s_2))$ & $\mathit{U}(p_{\star}^{\lambda_\star}, 10^5)$ \\ 
\hline 
  $3^7$ & $541$& $nn$ &$nn$ &$nn$ &$nn$ & $\lbrace \rbrace$ \\ 
\hline 
 $3^6$& $1093$ & $nn$ &$nn$ &$nn$ &$nn$ &$\lbrace \rbrace$ \\ 
\hline 
 $3^5$& $7$ & $13$ & $n$ &$n$ &$7$ & $\lbrace 3^5\cdot 7\cdot 13\rbrace$ \\ 
\hline 
 $3^4$& $11^2$ & $n$ & $7$ & $19$ & $nn$ & $\lbrace \rbrace$ \\ 
\hline 
 $3^3$& $5$ & $n$ &$n$ &$n$ &$n$ & $\lbrace 3^3\cdot 5, 3^3\cdot 5^2, $ \\ 

   & &  & & & & $ 3^3\cdot 5^3, 3^3\cdot 5^4, $ \\
  
  & &  & & & & $  3^3\cdot 5^5 \rbrace$ \\
\hline 
 $3^2$& $13$ & $n$ & $7$ &$n$ &$n$ &$\lbrace 3^2\cdot 13\cdot 7, $ \\

   & &  & & & & $ 3^2\cdot 13\cdot 7^2, $ \\ 
 
   & &  & & & & $ 3^2\cdot 13\cdot 7^3, $ \\ 
 
   & &  & & & & $ 3^2\cdot 13^2\cdot 7, $ \\
 
   & &  & & & & $ 3^2\cdot 13^2\cdot 7^2 \rbrace$ \\
\hline 
$^\star 3$ & $n$ &$n$  &$n$ &$n$ & $n$&  \\  
\hline 
$5^4$& $11$ & $71$ & $n$& $n$& $n$& $\lbrace \rbrace$ \\

\hline 
 $5^3$& $13$ & $n$ & $7$ &$n$ & $n$& $\lbrace5^3\cdot 13\cdot 7, $\\

 & &  & & & & $  5^3\cdot 13\cdot 7^2\rbrace$ \\  
\hline 
 $5^2$& $31$ & $n$ & $n$&$n$ &$n$ & $\lbrace5^2\cdot 31, 5^2\cdot 31^2 \rbrace$ \\ 
\hline 
 $^\star 5$& $n$ & $n$ &$n$ &$n$ &$n$ & \\ 
\hline 
 $7^3$& $5^2$ & $n$ & $11$ & $n$ &$nn$ & $\lbrace  \rbrace$ \\ 
\hline 
 $7^2$& $19$ & $n$ & $5$ & $n$ & $n$& $\lbrace  \rbrace$ \\ 
\hline 
 $^\star 7$ & $n$& $n$ &$n$ &$n$ &$n$ &  \\ 
\hline 
 $11^2$& $7$ & $19$ &$n$ & $n$& $5$& $\lbrace  \rbrace$ \\ 
\hline 
 $^\star 11$ & $n$& $n$ &$n$  &$n$ &$n$ & \\ 
\hline 
 $13^2$& $61$ & $n$ & $31$ &$n$ &$nn$ & $\lbrace  \rbrace$ \\ 
\hline 
 $^\star 13$ & $7$ & $n$ & $n$&$n$ &$n$ &  \\ 
\hline 
\end{tabular}
\\
\begin{flushleft}
Table 6.
\end{flushleft}
\vspace{1cm}
For the sake of small available space, we omitted column ${f}_1(\sigma(s_2))$ (which, of course, should be filled with 'nn' through). We also considered $p_\star^{\lambda_\star}$ in asterisk at the end of Table 5. 
The next table is actually meant to test the virtual alpha numbers that transformed to alpha number:\\ 
\begin{center}
 \begin{tabular}{|c|c|c|c|c|}
\hline 
 ${u}$ & $\prod_{({f}_j(\sigma(u)),u)=1} {f}_j(\sigma(u))$ & $\bar{n}=u\prod {f}_j(\sigma(u))$ & $2^m\mid n$ or $3^w\mid n$ & $\chi_\alpha(\bar{n})$ \\ 
\hline 
 $3^5\cdot 7\cdot 13$ & $$ & $$ & $m>0$ & $0$ \\ 
\hline 
  $3^3\cdot 5$ & $$ & $$ & $m>0$ & $0$ \\ 
\hline 
  $3^3\cdot 5^2$ &  & $$ & $m>0$ & $0$ \\ 
\hline 
  $3^3\cdot 5^3$ &  & $$ & $m>0$ & $0$ \\ 
\hline 
  $3^3\cdot 5^4$ & $11\cdot 71$ & $>10^5$ & $m>0$ & $0$ \\ 
\hline 
  $3^3\cdot 5^5$ & $7\cdot 31$ & $>10^5$ & $m>0$ & $0$ \\ 
\hline 
 $3^2\cdot 13 \cdot 7$  & $$ & $$ & $m>0$ & $0$ \\ 
\hline
  $3^2\cdot 13 \cdot 7^2$  & $$ & $$ & $m>0$ & $0$ \\ 
\hline
  $3^2\cdot 13 \cdot 7^3$  & $5^2\cdot31$ & $>10^5$ & $m>0$ & $0$ \\  
\hline
  $3^2\cdot 13^2\cdot 7$  & $61\cdot 31$ & $>10^5$ & $m>0$ & $0$ \\ 
\hline
  $3^2\cdot 13^2\cdot 7^2$  & $19\cdot 61\cdot 5$ & $>10^5$ & $m>0$ & $0$ \\ 
\hline  
 $5^3\cdot 13\cdot 7$ & $$ &  & $m>0$ & $0$ \\ 
\hline 
 $5^3\cdot 13\cdot 7^2$ & $19$ & $>10^5$ & $m>0$ & $0$ \\
\hline 
 $5^2\cdot 31$ & $$ & $$ & $m>0$ & $0$ \\ 
\hline 
  $5^2\cdot 31^2$ & $331\cdot83$ & $>10^5$ & $m\ge0$ & $0$ \\ 
\hline 
\end{tabular}
\\
\end{center} 
\end{center}
\begin{flushleft}
Table 7.
\end{flushleft}
\vspace{1cm}
Finally for $\lambda_\star=1$, $p_\star^{\lambda_\star}\in\lbrace 3,5,7,11,13\rbrace$. So there exist $u\in\mathit{U}(10^5, p_\star^{\lambda_\star})$ in Table 7 so that the new seed $u_\star=p_\star^{\lambda_\star}u$ with $(p_\star^{\lambda_\star},u)=1$. By our initial supposition that alpha number $n$ exists, $n=p_\star^{\lambda_\star}uz$, thus implies that $$\alpha_2=\alpha_1\frac{n}{\sigma(p_\star^{\lambda_\star})\sigma(u)\sigma(z)}.$$  But $u\prod {f}_j(\sigma(u))>10^5$ or $2\mid {\mathit{f}}(\sigma(u))$ and ${\mathit{f}}(\sigma(u))\mid n$ or $p_i\mid n$ with $p_i<p_\star$, a contradiction.
\\
\end{proof}
\section{{Conclusion and Recommendation:}}
The results of this paper, particularly, Theorems 3.2 , 3.3, 3.5 and 3.7 give a rough picture of the general form for odd @-numbers of order (1,1), if at all exist. So, in order to fully establish Conjecture 1 of this paper, it suffices to establish the case of non-square-free, non-prime-power odd $n$. This is recommended for further study (see \cite{{san1}} $\&$ \cite{{san2}} for motivation). Also note that
an in-depth study of alpha numbers can be pursued further as follows:
\begin{itemize} 
\item[(2)] Is there a general form for even alpha numbers analogous to Euclid-Euler form for even perfect numbers?
\item[(3)] What is the congruent form (properties) of odd alpha numbers analogous to Euler form for odd perfect numbers?
\item[(4)] Is every alpha number a practical number?
\item[(5)] What are the properties of odd alpha number $n$ (especially of those alpha numbers with non-rational complex order) in terms of size, the bounds (lower and upper), abundancy and etc?
\item[(6)] Is there any applicable relationship between function $\sigma_{\underline{\alpha}}(.)/(.)^{\bar{\alpha}}$ and the Riemann zeta function $\zeta(.)$?
\item[(7)] Is there any efficient and effective algorithm to generate alpha numbers?
\item[(8)] Can there be a counting function, say $C(x)$ generating the number of alpha number up to a
desired bound $x$ such that the number of alpha numbers in regular intervals, say $10^{0}-10^{3}$, $10^{3}-2\cdot  10^{3}$, $2\cdot 10^{3}-3\cdot  10^{3}$ and etc can be determined?
\item[(9)] Are the zeros of $\lfloor|\sigma_{\underline{\alpha}}(n)|\rfloor$ of complex order $\alpha$ in table 1 significant in any way?
\item[(10)] What other hidden identities of $\sigma_\alpha(.)$ function can be derived in order to solve
the problem of existence of alpha numbers?
\item[(11)] Let alpha number $n$ of order $(\underline{\alpha}, \bar{\alpha})$ with pair $(\alpha_1, \alpha_2)$ be denoted by $(\underline{\alpha}, \bar{\alpha}; \alpha_1, \alpha_2)- n$, and also let $\mathrm{T}_k^\alpha$ be a transformation on alpha numbers such that $$\mathrm{T}_k^\alpha : (\underline{\alpha}, \bar{\alpha}; \alpha_1, \alpha_2)- n \rightarrow(\underline{\alpha}, \bar{\alpha}; \alpha_1, \alpha_2)- m=n p_1\cdot p_2\cdots p_k.$$ For example, $$(1,1; 13,4)-2^9\cdot 3^2\cdot 31\cdot 11\  _{\overrightarrow{\mathrm{T}_1^\alpha}} \ (1,1; 7,2)-2^9\cdot 3^2\cdot 31\cdot 11\cdot 13  $$ $$ (1,1; 13,4)-2^9\cdot 3^2\cdot 31\cdot 11\  _{\overrightarrow{\mathrm{T}_2^\alpha}} \ (1,1; 4,1)-2^9\cdot 3^2\cdot 31\cdot 11\cdot 13\cdot 7$$ which transformed a weak alpha number $ 2^9\cdot 3^2\cdot 31\cdot 11 $ to strong alpha number $2^9\cdot 3^2\cdot 31\cdot 11\cdot 13\cdot 7.$ Then one might wish to determine how prevalent do such transformation occur and terminate.
\item[(12)] What formidable results can come forth from the following certain generalizations of alpha numbers (see  \cite{{laa}}\cite{{san1}} \cite{san2} $\&$ \cite{{wei}} for definitions, notations and motivation)?
\end{itemize}
\begin{itemize}
\item[(I)] Let any positive integer $n$ satisfying $\sigma^\alpha_{\underline{\alpha}}(n)=(\alpha_1/\alpha_2)n^{\bar{\alpha}}$ be called $\alpha$-super @-number of order $(\underline{\alpha},\bar{\alpha})$, where $\sigma^\alpha_{\underline{\alpha}}(n)=\sigma^{\alpha-1}_{\underline{\alpha}}\sigma_{\underline{\alpha}}(n)$ is the divisor function such that $\sigma^\alpha$ is the $\alpha$th iterate of $\sigma$-function and co-prime pair integral $1 \leq \max(\alpha_1, \alpha_2)\leq \omega(n)$, $\omega(n)< \max(\alpha_1, \alpha_2)\leq \tau(n)$ and $\tau(n)< \max(\alpha_1, \alpha_2) < n$ respectively for strong, weak, very weak of such number.
\item[(II)] Let any positive integer $n$ satisfying $\sigma^\star_{\underline{\alpha}}(n) =(\alpha_1/\alpha_2)n^{\bar{\alpha}}$ be called $m$-unitary @-number, where $\sigma^\star_\alpha(n)$ is the unitary integral divisor function such that unitary divisors of $n$ are used instead of positive divisors of $n$ in the computation of $\sigma_k(n)$.
\item[(III)]Let any positive integer $n$ satisfying $\sigma_{\alpha,\infty}(n)=(\alpha_1/\alpha_2)n^{\bar{\alpha}}$ be called $m$-infinitary @-number, where $\sigma_{\alpha,\infty}(n)$ is infinitary integral divisor function such that infinitary divisors of $n$ are used instead of positive divisors of $n$ in the computation of $\sigma_{\underline{\alpha}}(n)$.
\item[(IV)] Let any positive integer $n$ satisfying $\sigma_{\alpha
,e}(n))=(\alpha_1/\alpha_2)n^{\bar{\alpha}}$ be called m-exponential @-number, where $\sigma_{\alpha,e}(n)$ is the integral exponential divisor function such that exponential divisors of $n$ are used instead of positive divisors of $n$ in the computation of $\sigma_{\underline{\alpha}}(n)$.
\item[(V)] Let any positive integer $n$ satisfying $\sigma_{\underline{\alpha}}(n))=\frac{\alpha_1}{\alpha_2}(n+k)^{\bar{\alpha}}:k<\omega(n)$ be called near @-number, where $\sigma_{\alpha}(n)$ is the sum of divisor function.
\end{itemize}
 
\textbf{Acknowledgement:}
Thanks to Professor V. A. Babalola, Professor Terence Tao, and Professor J. S$\acute{a}$ndor for their timely support concerning arXiv documentation of this paper.

\end{document}